\documentclass[11pt]{article}
\usepackage[T1]{fontenc}
\usepackage{tocbibind}
\usepackage{amsmath,amsthm,amssymb}
\usepackage{fancyhdr}
\usepackage[british]{babel}
\usepackage{geometry,mathtools}
\usepackage{enumitem}
\usepackage{algpseudocode}
\usepackage{dsfont}
\usepackage{centernot}
\usepackage{xstring}
\usepackage{colortbl}
\usepackage[section]{algorithm}
\usepackage{graphicx}
\usepackage[font={footnotesize}]{caption}
\usepackage[usenames,dvipsnames,table]{xcolor}
\usepackage{pstricks,tikz}
\usepackage[h]{esvect}
\usepackage[
		bookmarksopen=true,
		bookmarksopenlevel=1,
		colorlinks=true,
		linkcolor=darkblue,
        linktoc=page,
		citecolor=darkblue,
]{hyperref}
\usepackage{bbm}

\numberwithin{equation}{section}

\definecolor{darkblue}{rgb}{0,0,0.5}

\newdimen\margin
\def\textno#1&#2\par{
   \margin=\hsize
   \advance\margin by -4\parindent
          \setbox1=\hbox{\sl#1}
   \ifdim\wd1 < \margin
      $$\box1\eqno#2$$
   \else
      \bigbreak
      \hbox to \hsize{\indent$\vcenter{\advance\hsize by -3\parindent
      \it\noindent#1}\hfil#2$}
      \bigbreak
   \fi}

\newtheorem{theorem}[algorithm]{Theorem}
\newtheorem{prop}[algorithm]{Proposition}
\newtheorem{lemma}[algorithm]{Lemma}

\theoremstyle{definition}

\def\lateproof#1{\removelastskip\penalty55\medskip\noindent\begin{stepenv}\end{stepenv}{\bf Proof of #1. }} % in each main proof, claim and step counter set back

\def\noproof{{\unskip\nobreak\hfill\penalty50\hskip2em\hbox{}\nobreak\hfill%
       $\square$\parfillskip=0pt\finalhyphendemerits=0\par}\goodbreak}
\def\endproof{\noproof\bigskip}

\newcounter{stepenv}
\newenvironment{stepenv}[1][]{\refstepcounter{stepenv}}{}

\newcounter{step}[stepenv]

\newcounter{substep}[step]
\renewcommand{\thesubstep}{\thestep.\arabic{substep}}

\newcounter{claim}[stepenv]

\newcommand{\cE}{\mathcal{E}}
\newcommand{\cF}{\mathcal{F}}

\newcommand{\cP}{\mathcal{P}}

\newcommand{\bN}{\mathbb{N}}

\newcommand{\bR}{\mathbb{R}}

\def\eps{{\epsilon}}
\newcommand{\eul}{{e}}

\newcommand{\defn}{\emph}

\newcommand{\prob}[1]{\mathrm{\mathbb{P}}\left[#1\right]}
\newcommand{\cprob}[2]{\prob{#1 \;\middle|\; #2}}
\newcommand{\expn}[1]{\mathrm{\mathbb{E}}\left[#1\right]}

\def\sm{\setminus}
\newcommand{\Set}[1]{\{#1\}}
\newcommand{\set}[2]{\{#1\,:\;#2\}}
\def\In{\subset}

\newcommand{\IND}{\mathbbm{1}}

\def\COMMENT#1{}
\def\TASK#1{}
\let\TASK=\footnote             % COMMENT OUT for clean output
\begin{document}

\title{Note on induced paths in sparse random graphs}

\author{Stefan Glock \thanks{Institute for Theoretical Studies, ETH, 8092 Z\"urich, Switzerland.
Email: \href{mailto:dr.stefan.glock@gmail.com}{\nolinkurl{dr.stefan.glock@gmail.com}}. Research supported by Dr.~Max R\"ossler, the Walter Haefner Foundation and the ETH Z\"urich Foundation.}
}

\date{}

\maketitle

\begin{abstract} 
We show that for $d\ge d_0(\eps)$, with high probability, the random graph $G(n,d/n)$ contains an induced path of length $(3/2-\eps)\frac{n}{d}\log d$. This improves a result obtained independently by \L{}uczak and Suen in the early 90s, and answers a question of Fernandez de la Vega. 
Along the way, we generalize a recent result of Cooley, Dragani\'c, Kang and Sudakov who studied the analogous problem for induced matchings.
\end{abstract}

\section{Introduction}

Let $G(n,p)$ denote the binomial random graph on $n$ vertices, where each edge is included independently with probability~$p$. In this note, we are concerned with \emph{induced} subgraphs of $G(n,p)$, specifically trees and paths.

The study of induced trees in $G(n,p)$ was initiated by Erd\H{o}s and Palka~\cite{EP:83} in the 80s. Among other things, they showed that for constant $p$, with high probability (\textbf{whp}) the size of a largest induced tree in $G(n,p)$ is asymptotically equal to $2\log_q(np)$, where $q=\frac{1}{1-p}$. The obtained value coincides asymptotically with the \emph{independence number} of $G(n,p)$, the study of which dates back even further to the work of Bollob\'as and Erd\H{o}s~\cite{BE:76}, Grimmett and McDiarmid~\cite{GM:75} and Matula~\cite{matula:76}.

As a natural continuation of their work, Erd\H{o}s and Palka~\cite{EP:83} posed the problem of determining the size of a largest induced tree in \emph{sparse} random graphs, when $p=d/n$ for some fixed constant~$d$. More precisely, they conjectured that for every $d>1$ there exists $c(d)>0$ such that \textbf{whp} $G(n,p)$ contains an induced tree of order at least $c(d)\cdot n$.
This problem was settled independently in the late 80s by Fernandez de la Vega~\cite{fernandez-de-la-Vega:86}, Frieze and Jackson~\cite{FJ:87a}, Ku\v{c}era and R\"{o}dl~\cite{KR:87} as well as \L{}uczak and Palka~\cite{LP:88}.
%\COMMENT{in~\cite{LP:88}, de la Vega and Frieze--Jackson are already cited, and there is no claim that it is independent, but a new proof. But Luczak himself in his later paper lists all 4 papers as independent, and so do other later papers}
In particular, Fernandez de la Vega~\cite{fernandez-de-la-Vega:86} showed that one can take $c(d)\sim \frac{\log d}{d}$, and a simple first moment calculation reveals that this is tight within a factor of~$2$.

Two natural questions arise from there. First, one might wonder whether it is possible to find not only some  \emph{arbitrary} induced tree, but a \emph{specific} one, say a long induced path. Indeed, Frieze and Jackson~\cite{FJ:87b} in a separate paper showed that \textbf{whp} there is an induced path of length $\tilde{c}(d)\cdot n$. Two weaknesses of this result were that their proof only worked for sufficiently large~$d$, and that the value obtained for $\tilde{c}(d)$ was far away from the optimal one.
Later, \L{}uczak~\cite{luczak:93} and Suen~\cite{suen:92} independently remedied this situation twofold. They proved that an induced path of length linear in $n$ exists for all $d>1$, showing that the conjecture of Erd\H{o}s and Palka holds even for induced paths. Moreover, they showed that one can take $\tilde{c}(d)\sim \frac{\log d}{d}$ as in the case of arbitrary trees. 

A second obvious question is to determine the size of a largest induced tree (and path) more precisely. The aforementioned results were proved by analysing the behaviour of certain constructive algorithms which produce large induced trees and paths. The value $\frac{\log d}{d}$ seems to constitute a natural barrier for such approaches. On the other hand, recall that in the dense case, the size of a largest induced tree coincides asymptotically with the independence number. In 1990, Frieze~\cite{frieze:90} showed that the first moment bound $\sim2\frac{n}{d}\log d$ is tight for the independence number, also in the sparse case. His proof is based on the profound observation that the second moment method can be used even in situations where it apparently does not work, if one can combine it with a strong concentration inequality.
Finally, in 1996, Fernandez de la Vega~\cite{fernandez-de-la-Vega:96} observed that the earlier achievements around induced trees can be combined with Frieze's breakthrough to prove that the size of a largest induced tree is indeed $\sim 2\frac{n}{d}\log d$. This complements the result of Erd\H{o}s and Palka~\cite{EP:83} in the dense case. (When $p=o_n(1)$, we have $2\log_q(np)\sim 2\frac{n}{d}\log d$.)

Fernandez de la Vega~\cite{fernandez-de-la-Vega:96} also posed the natural problem of improving the \L{}uczak--Suen bound~\cite{luczak:93,suen:92} for induced paths, for which his approach was ``apparently helpless''. Despite the widely held belief (see~\cite{CDKS:ta,DS:18} for instance) that the upper bound $\sim 2\frac{n}{d}\log d$ obtained via the first moment method is tight, the implicit constant $1$ has not been improved in the last 30 years.

\subsection{Long induced paths}\label{sec:paths}

Our main result is the following, which solves the problem of Fernandez de la Vega. Unfortunately, we only get ``halfway'' towards the optimal bound, and the obvious problem left open is to close the remaining gap.

\begin{theorem}\label{thm:path}
For any $\eps>0$ there is $d_0$ such that \textbf{whp} $G(n,p)$ contains an induced path of length $(3/2-\eps)\frac{n}{d}\log d$ whenever $d_0\le d=pn =o(n)$.
\end{theorem}

For the sake of generality, we state our result for a wide range of functions~$d=d(n)$. 
However, we remark that the most interesting case is when $d$ is a sufficiently large constant, and any improvement in this regime is likely to generalize straightforwardly. 
In fact, for dense graphs, when $d\ge n^{1/2}\log^2 n$, much better results are already known (cf.~\cite{DS:18,rucinski:87}).

Some of the earlier results~\cite{FJ:87b,luczak:93} are phrased in terms of induced cycles (\defn{holes}). Using a simple sprinkling argument, one can see that aiming for a cycle instead of a path does not make the problem any harder.

We also note that our proof is self-contained, except for well-known facts from probability and graph theory.

We briefly explain our strategy, and also discuss how the approach might be used to eventually match the upper bound.
 The idea is to find a long induced path in two steps. First, we find many disjoint paths of some chosen length $L$, such that the subgraph consisting of their union is induced. To achieve this, we generalize a recent result of Cooley, Dragani\'c, Kang and Sudakov~\cite{CDKS:ta} who obtained large induced matchings. We will discuss this further in Section~\ref{sec:forests}.
Assuming now we can find such an induced linear forest $F$, the aim is to connect almost all of the small paths into one long induced path, using a few additional vertices. (In order to maintain randomness for the connection step, we only expose half the vertices to find~$F$.)
To model this, we give each path in $F$ a direction, and define an auxiliary digraph whose vertices are the paths, and two paths $(P_1,P_2)$ form an edge if there exists a new ``connecting'' vertex $a$ that has some edge to the last $\eps L$ vertices of $P_1$ and some edge to the first $\eps L$ vertices of $P_2$, but no edge to the rest of~$F$.
Our goal is to find an almost spanning path in this auxiliary digraph. Observe that this will provide us with a path in $G(n,p)$ of length roughly~$|F|$. Moreover, if we can ensure that the new connecting vertices form an independent set, this path will be induced.
The intuition is that the auxiliary digraph behaves quite randomly, which gives us hope that, even though it is very sparse, we can find an almost spanning path. In order to back this up and illustrate the interplay between the above parameters, let us assume that in the first step we can find an induced linear forest $F$ of order $c\frac{n}{d}\log d=cp^{-1}\log d$ with components of order $L\approx d^{\alpha}$, where $c$ and $\alpha$ are constants to be specified later.
Consider now two paths $P_1,P_2$ in~$F$. For a new vertex $a$, the probability that it joins to the desired segments of $P_1,P_2$ is $\approx (Lp)^2$, and the probability that it has no edge to the rest of $F$ is at least $(1-p)^{|F|}\approx \exp(-p|F|) = d^{-c}$. Since there are $\Theta(n)$ potential connecting vertices, we estimate the probability that $(P_1,P_2)$ is an edge in the auxiliary digraph to $\approx L^2p^2d^{-c}n$. Noting that the order of our digraph is $N= c\frac{n}{dL}\log d$, we infer that its average degree is $\approx d^{1-c}L\log d$. 
It is well-known in random graph theory that a random $N$-vertex digraph contains a path of length $(1-\eps)N$ if the average degree is sufficiently large as a function of~$\eps$.
This suggests that our strategy could work if $c\le 1+\alpha$. Indeed, it turns out that we can ensure $L\approx d^{1/2}$ in the first step of our argument (see Lemma~\ref{lem:forests}), hence the constant $3/2$ in Theorem~\ref{thm:path}. 
If one could ensure that $L$ is almost linear in~$d$, that is, $\alpha\approx 1$, then our argument would directly yield an induced path of the asymptotically optimal length $\sim 2\frac{n}{d}\log d$. The proof of the connecting step will be given in Section~\ref{sec:connect}.

\subsection{Induced forests with small components}\label{sec:forests}

As outlined above, in the first step of our argument, we seek an induced linear forest whose components are paths of length $L\approx d^{1/2}$.
For this, we generalize a recent result of Cooley, Dragani\'c, Kang and Sudakov~\cite{CDKS:ta}. They proved that \textbf{whp} $G(n,p)$ contains an induced matching with $\sim 2\log_q (np)$ vertices, which is asymptotically best possible. They also anticipated that using a similar approach one can probably obtain induced forests with larger, but bounded components. As a by-product, we confirm this.
To state our result, we need the following definition. For a given graph $T$, a \defn{$T$-matching} is a graph whose components are all isomorphic to~$T$. Hence, a $K_2$-matching is simply a matching, and the following for $T=K_2$ implies the main result of~\cite{CDKS:ta}. 

\begin{theorem}\label{thm:forests}
	For any $\eps>0$ and tree $T$, there exists $d_0>0$ such that \textbf{whp} the order of the largest induced $T$-matching in $G(n,p)$ is $(2\pm \eps)\log_q(np)$, where $q=\frac{1}{1-p}$, whenever $\frac{d_0}{n}\le p\le 0.99$.
\end{theorem}

We use the same approach as in~\cite{CDKS:ta}, which goes back to the work of Frieze~\cite{frieze:90} (see also~\cite{bollobas:88,SS:87}).
The basic idea is as follows. Suppose we have a random variable $X$ and want to show that \textbf{whp}, $X\ge b-t$, where $b$ is some ``target'' value and $t$ a small error. 
For many natural variables, we know that $X$ is ``concentrated'', say $\prob{|X- \expn{X}| \ge  t/2} < \rho$ for some small~$\rho$. This is the case for instance when $X$ is determined by many independent random choices, each of which has a small effect.
However, it might be difficult to estimate $\expn{X}$ well enough.
But if we know in addition that $\prob{X\ge b}\ge \rho$, then we can combine both estimates to $\prob{X\ge b}> \prob{X\ge \expn{X}+t/2}$, which clearly implies that $b\le \expn{X}+t/2$. Applying now the other side of the concentration inequality, we infer $\prob{X\le b-t} \le \prob{X\le \expn{X}-t/2}< \rho$, as desired.

In our case, say $X$ is the maximum order of an induced $T$-matching in~$G(n,p)$. Since adding or deleting edges at any one vertex can create or destroy at most one component, we know that $X$ is $|T|$-Lipschitz and hence concentrated (see Section~\ref{sec:concentration}). Using the above approach, it remains to complement this with a lower bound on the probability that $X\ge b$. Introduce a new random variable $Y$ which is the \defn{number} of induced $T$-matchings of order~$b$ (a multiple of $|T|$). Then we have $X\ge b$ if and only if $Y>0$. The main technical work is to obtain a lower bound for the probability of the latter event using the second moment method.
We note that by applying the second moment method to labelled copies (instead of unlabelled copies as in~\cite{CDKS:ta}) we obtain a shorter proof even in the case of matchings (see Section~\ref{sec:2ndmoment}).
More crucially, it turns out that one can even find induced forests where the component sizes can grow as a function of~$d$. As discussed above, for the proof of Theorem~\ref{thm:path} we need an induced linear forest where the components are paths of length roughly $d^{1/2}$. This is provided by the following auxiliary result. We note that the same holds for forests with arbitrary components of bounded degree, and one can also let the degree slowly grow with~$d$, but we choose to keep the presentation simple. 
%Note that this only makes sense if $d\ll n^{2/3}\log n$, which is why we state it separately from Theorem~\ref{thm:forests}

\begin{lemma}\label{lem:forests}
	For any $\eps>0$, there exists $d_0>0$ such that \textbf{whp} $G(n,p)$ contains an induced linear forest of order at least $(2-\eps)p^{-1}\log(np)$ and component paths of order $d^{1/2}/\log^4 d$, whenever $d_0\le d=np \le n^{1/2}\log^2 n$. 
\end{lemma}

It would be interesting to find out whether the length of the paths can be improved to $d^{1-o(1)}$. As remarked earlier, this would lead to the asymptotically optimal result for the longest induced path problem.

\subsection{Notation}

We use standard graph theoretical notation. In particular, for a graph $G$ and $U\In V(G)$, we let $e(G)$ denote the number of edges, $\Delta(G)$ the maximum degree and $G[U]$ the subgraph induced by~$U$.
Recall that a forest is called \defn{linear} if its components are paths. 

For functions $f(n),g(n)$, we write $f\sim g$ if $\lim_{n\to\infty}\frac{f(n)}{g(n)}=1$. We also use the standard Landau symbols $o(\cdot),\Omega(\cdot),\Theta(\cdot),O(\cdot),\omega(\cdot)$, where subscripts disclose the variable that tends to infinity if this is not clear from the context.
We use $\approx$ non-rigorously in informal discussions and ask the reader to interpret it correctly.

An event $\cE_n$ holds \defn{with high probability} (\textbf{whp}) if $\prob{\cE_n}= 1-o_n(1)$.
We use $\log$ to denote the natural logarithm with base~$\eul$. Moreover, $[n]=\Set{1,\dots,n}$ and $(n)_k=n(n-1)\cdots (n-k+1)$.
Recall the standard estimates $\binom{n}{k}\le \left(\frac{en}{k}\right)^k$, $1+x\le \eul^x$ and $\log(1+x)=x+O\left(\frac{x^2}{1-|x|}\right)$, where the latter holds for $|x|<1$ and implies $1-x\ge \eul^{-x-O(x^2)}$ for $0\le x\le 0.99$, say.
As customary, we tacitly treat large numbers like integers whenever this has no effect on the argument.

\section{Second moment}\label{sec:2ndmoment}

In this section, we use the second moment method to derive a lower bound on the probability that $G(n,d/n)$ contains a given induced linear forest of size $\sim 2\frac{n}{d}\log d$. Here, it does not matter that the components are small.
More precisely, we prove that for fixed $\eps>0$ and $d\ge d_0(\eps)$, \emph{any} bounded degree forest $F$ on $k\le (2-\eps)\frac{n}{d}\log d$ vertices is an induced subgraph of $G(n,d/n)$ with probability at least $\exp(-O(\frac{n\log^2 d}{d^2}))$.
Moreover, when $d=\omega(n^{1/2}\log n)$, the obtained probability bound tends to~$1$. In particular, in this regime, the lemma readily implies the existence of an induced path of the asymptotically optimal length $\sim 2\frac{n}{d}\log d$ \textbf{whp}.

\begin{lemma}\label{lem:2ndmoment}
	For any $\eps>0$, there exists $d_0$ such that the following holds for all $d_0\le d < n$, where $p=\frac{d}{n}$ and $q=\frac{1}{1-p}$. For any forest $F$ on $k\le (2-\eps)\log_q d$ vertices with maximum degree $\Delta\le d^{\eps/6}$, the probability that $G(n,p)$ contains an induced copy of $F$ is at least $$\exp\left(-10^4\Delta^2\frac{n\log^2 d}{d^2} -2d^{-\eps/7}\right) .$$
\end{lemma}

The proof of Lemma~\ref{lem:2ndmoment} is based on the second moment method and will be given below. We start off with some basic preparations which will also motivate the main counting tool.

Fix a forest~$F$ of order~$k$. Let $Y$ be the random variable which counts the number of \emph{labelled} induced copies of $F$ in $G(n,p)$. More formally, let $\cF$ be the set of all injections $\sigma\colon V(F)\to [n]$, and for $\sigma\in \cF$, let $F_\sigma$ be the graph with vertex set $\set{\sigma(x)}{x\in V(F)}$ and edge set $\set{\sigma(x)\sigma(y)}{xy\in E(F)}$.
Let $A_\sigma$ be the event that $F_\sigma$ is an induced subgraph of~$G(n,p)$.
Hence, $$\prob{A_\sigma}=p^{e(F)}(1-p)^{\binom{k}{2}-e(F)},$$
and setting $Y=\sum_{\sigma\in \cF}\IND(A_\sigma)$, we have 
\begin{align}
\expn{Y}=(n)_k p^{e(F)}(1-p)^{\binom{k}{2}-e(F)}.\label{expn}
\end{align}
Ultimately, we want to obtain a lower bound for $\prob{Y>0}$. 
Fix some $\sigma_0\in \cF$. By symmetry, the second moment of $Y$ can be written as $$\expn{Y^2}=\expn{Y} \sum_{\sigma\in \cF}\cprob{A_\sigma}{A_{\sigma_0}}.$$
Applying the Paley--Zygmund inequality, we thus have 
\begin{align}
\prob{Y>0}\ge \frac{\expn{Y}^2}{\expn{Y^2}} = \frac{\expn{Y}}{\sum_{\sigma\in \cF}\cprob{A_\sigma}{A_{\sigma_0}}}.\label{2ndmoment inequality}
\end{align}

The remaining difficulty is to control the terms $\cprob{A_\sigma}{A_{\sigma_0}}$.
We say that $\sigma\in \cF$ is \defn{compatible} (with $\sigma_0$) if $\cprob{A_\sigma}{A_{\sigma_0}}>0$. This means that, in the intersection $V(F_\sigma)\cap V(F_{\sigma_0})$, a pair $uv$ which is an edge in $F_\sigma$ cannot be a non-edge in $F_{\sigma_0}$, and vice versa, as otherwise $F_\sigma$ and $F_{\sigma_0}$ could not be induced subgraphs of $G(n,p)$ simultaneously. From now on, we can ignore all 
$\sigma$ that are not compatible with $\sigma_0$.

If $\sigma\in\cF$ is compatible with $\sigma_0$, we denote by $I_\sigma:=F_\sigma \cap F_{\sigma_0}$ the graph on $S=V(F_\sigma)\cap V(F_{\sigma_0})$ with edge set $E(F_\sigma[S])=E(F_{\sigma_0}[S])$. This ``intersection graph'' assumes a crucial role in the analysis.
Suppose that $I_\sigma$ has $s$ vertices and $c$ components. Since $I_\sigma$ is a forest, we have $e(I_\sigma)=s-c$. These are the edges of $F_\sigma$ that we already know to be there when conditioning on~$A_{\sigma_0}$, and for $F_\sigma$, we need $e(F)-e(I_\sigma)$ ``new'' edges. 
Moreover, there are $\binom{k}{2}-\binom{s}{2}-e(F)+e(I_\sigma)$ additional non-edges. 
Therefore,
\begin{align}
\cprob{A_\sigma}{A_{\sigma_0}} = p^{e(F)-s+c} (1-p)^{\binom{k}{2}-\binom{s}{2}-e(F)+s-c}.\label{cond prob}
\end{align}
Note here that when the number of components $c$ is large, then the exponent of $p$ is large and hence we have a stronger upper bound on $\cprob{A_\sigma}{A_{\sigma_0}}$. On the other hand, if $c$ is small, then $\cprob{A_\sigma}{A_{\sigma_0}}$ is larger, but this will be compensated by the fact that there are fewer such~$\sigma$.
In the following, we bound the number of compatible $\sigma\in\cF$ for which $I_\sigma$ has $s$ vertices and $c$ components. We remark that this kind of analysis was also carried out in~\cite{draganic:20} in the study of dense random graphs. We include the details for completeness, with an improved dependence on~$\Delta$.
We make use of the following elementary counting result.

\begin{prop}\label{prop:branching}
For a graph $H$ with $\Delta(H)\le \Delta$ and $v\in V(H)$, the number of (unlabelled) trees in $H$ of order $s$ which contain $v$ is at most $(e\Delta)^{s-1}$.
\end{prop}

In the case that is relevant for our application, namely when $\Delta=2$, this is trivially true. The more general case follows easily from the formula for the number of rooted subtrees of a given order in the $\Delta$-regular infinite tree (see~\cite{KNP:20}).

\begin{prop}\label{prop:counting extension}
For all $0\le c\le s$, the number of compatible $\sigma\in\cF$ for which $I_\sigma$ has $s$ vertices and $c$ components is at most $$\binom{k}{c} k^c (6\Delta^2)^s (n-k)_{k-s}.$$
\end{prop}

\begin{proof}
Fix $s$ and~$c$. We can obviously assume that $s\ge c\ge 1$, as otherwise the bound is easily seen to hold.
The first claim is that the number of subgraphs of $F_{\sigma_0}$ with $s$ vertices and $c$ components is at most $\binom{k}{c} (2e\Delta)^s$. 
To see this, we first choose root vertices $v_1,\dots,v_c$ for the components, for which there are at most $\binom{k}{c}$ choices. For $i\in [c]$, let $T_i$ denote the component which will contain~$v_i$. Next, we fix the sizes of the components. Writing $s_i=|T_i|$, the number of possibilities is given by the number of positive integer solutions of $s_1+\dots+s_c=s$, which is $\binom{s-1}{c-1}\le 2^s$ by a well-known formula. 
Now, having fixed the sizes, we can apply Proposition~\ref{prop:branching} for each $i\in[c]$, with $F,v_i$ playing the roles of $H,v$, to see that the number of choices for $T_i$ is at most $(e\Delta)^{s_i-1}$, which combined amounts to $(e\Delta)^{s-c}$.
This implies the claim, and immediately yields an upper bound on the number of possibilities for the intersection graph~$I_\sigma$.

Now, fix a choice of $I_\sigma$. Since $I_\sigma$ is a forest with $c$ components, its vertices can be ordered such that every vertex, except for the first $c$ vertices, has exactly one neighbour preceding it.
In order to count the number of possibilities for $\sigma$, we proceed as follows. 
First, choose the preimages under $\sigma$ for the first $c$ vertices, for which there are at most $(k)_c$ choices. Now, we choose the preimages of the remaining vertices of $I_\sigma$ one-by-one in increasing order. In each step, there are at most $\Delta$ choices, since one neighbour of the current vertex has already chosen its preimage, and $I_\sigma$ has to be an induced subgraph of $F_\sigma$. Hence, there are at most $\Delta^{s-c}$ choices for the preimages of the remaining vertices of~$I_\sigma$.
Finally, we have used $s$ vertices of $F$ as preimages for the vertices in $I_\sigma$. The remaining $k-s$ vertices of $F$ must be mapped to $[n]\sm V(F_{\sigma_0})$, so there are at most $(n-k)_{k-s}$ possibilities.
\end{proof}

With the preparations done, the proof of the lemma reduces to a chain of estimates.

\lateproof{Lemma~\ref{lem:2ndmoment}}
By~\eqref{2ndmoment inequality}, it suffices to show that
\begin{align*}
	\frac{\sum_{\sigma\in \cF}\cprob{A_\sigma}{A_{\sigma_0}}}{\expn{Y}} \le \exp\left(10^4\Delta^2\frac{n\log^2 d}{d^2} + 2d^{-\eps/7}\right).
\end{align*}

We split the sum over compatible $\sigma\in \cF$ according to the number of vertices and components of~$I_\sigma$. Applying \eqref{expn},~\eqref{cond prob} and Proposition~\ref{prop:counting extension}, we obtain
\begin{align*}
\frac{\sum_{\sigma\in \cF}\cprob{A_\sigma}{A_{\sigma_0}}}{\expn{Y}} &\le
\sum_{s=0}^k \sum_{c=0}^s \frac{\binom{k}{c}k^c (6\Delta^2)^s (n-k)_{k-s}  p^{e(F)-s+c} (1-p)^{\binom{k}{2}-\binom{s}{2}-e(F)}}{(n)_k p^{e(F)}(1-p)^{\binom{k}{2}-e(F)}} \\
&= \sum_{s=0}^k \frac{(n-k)_{k-s}}{(n)_k}p^{-s}q^{\binom{s}{2}}(6\Delta^2)^s  \sum_{c=0}^s \binom{k}{c} (kp)^c \\
&\le  \sum_{s=0}^k (4/n)^s p^{-s} q^{s^2/2} (6\Delta^2)^s \frac{(16k\log d)^s}{s!}\\
&= \sum_{s=0}^k \frac{\left(\frac{384\Delta^2 k \log d}{d} q^{s/2}  \right)^s}{s!}.
\end{align*}
To verify the last inequality, note that we always have $k\le 2\log_q(np)\le 2\frac{n}{d}\log d$ since $\log(q)\ge p$. Hence, $kp\le 2\log d$. Moreover, we have
$\frac{(n-k)_{k-s}}{(n)_k} \le \frac{1}{(n)_s} \le (4/n)^s$ since $s\le k \le 2n/\eul$. We also used the fact that $\binom{k}{c} \le 4^s \binom{k}{s}\le  \frac{(4k)^s}{s!}$. To see this, observe that when $s\le k/2$, we have $\binom{k}{c}\le \binom{k}{s}$, and otherwise, $\binom{k}{c}\le 2^k \le 2^{2s}$. Finally, $c$ takes only $s+1\le 2^s$ values.

We split the final sum into two terms.
First, consider the range $s\le k/\log d$. Then $q^{s/2}\le q^{1/\log q}= \eul$. Hence, recalling the power series $e^x=\sum_{s\ge 0}\frac{x^s}{s!}$, we obtain the bound
$$\sum_{s=0}^{\lfloor k/\log d\rfloor } \frac{\left(\frac{384\Delta^2 k \log d}{d} q^{s/2}  \right)^s}{s!}\le \exp\left(\frac{384\eul \Delta^2 k\log d}{d} \right) \le \exp\left(\frac{10^4 \Delta^2 n\log^2 d}{d^2} \right).$$ 

Finally, for $s\ge k/\log d$, we use $s!\ge (s/\eul)^s$ to bound each summand as
\begin{align*}
\frac{\left(\frac{384\Delta^2 k \log d}{d} q^{s/2}  \right)^s}{s!} \le \left(\frac{384\eul \Delta^2 k\log d}{ds} q^{s/2} \right)^s \le \left(\frac{384\eul \Delta^2 \log^2 d}{d}  q^{s/2} \right)^s.
\end{align*}
Crucially, since $s\le k\le (2-\eps)\log_q d$, we have $q^{s/2} \le q^{(1-\eps/2)\log_q d} = d^{1-\eps/2}$. Now, for sufficiently large $d\ge d_0$ the bracket is bounded by $\frac{384\eul \Delta^2 \log^2 d}{d^{\eps/2}} \le d^{-\eps/7}<1$. Therefore the geometric series tells us that 
$$\sum_{s=\lceil k/\log d \rceil}^{k} \frac{\left(\frac{384\Delta^2 k \log d}{d} q^{s/2}  \right)^s}{s!}  \le \frac{1}{1-d^{-\eps/7}} -1 \le 2d^{-\eps/7}.$$ 

Altogether, we conclude that
\begin{align*}
\frac{\sum_{\sigma\in \cF}\cprob{A_\sigma}{A_{\sigma_0}}}{\expn{Y}} \le  \exp\left(\frac{10^4 \Delta^2 n\log^2 d}{d^2} \right) + 2d^{-\eps/7} \le \exp\left(\frac{10^4 \Delta^2 n\log^2 d}{d^2} + 2d^{-\eps/7}\right),
\end{align*}
completing the proof.
\endproof

\section{Concentration}\label{sec:concentration}

In this section, we deduce Theorem~\ref{thm:forests} and Lemma~\ref{lem:forests} from Lemma~\ref{lem:2ndmoment}. 
We will use Talagrand's inequality.\COMMENT{One could also make the argument work with Azuma's inequality, but Talagrand's inequality is more convenient to use.}
To state it, we need the following definitions.
Given a product probability space $\Omega=\prod_{i=1}^n \Omega_i$ (endowed with the product measure) and a random variable $X\colon \Omega\to \bR$, we say that $X$ is
\begin{itemize}
	\item \defn{$L$-Lipschitz} (for some $L>0$) if for any $\omega,\omega'\in \Omega$ which differ only in one coordinate, we have $|X(\omega)-X(\omega')|\le L$;
	\item \defn{$f$-certifiable} (for a function $f\colon \bN\to \bN$) if for every $s$ and $\omega$ such that $X(\omega)\ge s$, there exists a set $I\In [n]$ of size $\le f(s)$ such that $X(\omega')\ge s$ for every $\omega'$ that agrees with $\omega$ on the coordinates indexed by~$I$.
\end{itemize}

\begin{theorem}[Talagrand's inequality, see~\cite{AS:08}]
	Suppose that $X$ is $L$-Lipschitz and $f$-certifiable. Then, for all $b,t\ge 0$,
	$$\prob{X\le b-tL\sqrt{f(b)}} \prob{X\ge b} \le \exp\left(-t^2/4\right).$$
\end{theorem}

Our probability space is of course $G(n,p)$. Although this comes naturally as a product of $\binom{n}{2}$ elementary probability spaces $\Omega_{ij}$, one for each potential edge $ij$, it can be more effective, depending on the problem, to consider a description that is vertex-oriented, where the edges incident to a vertex are combined into one probability space.
Concretely, for $i\in [n-1]$, let $\Omega_i=\prod_{j>i}\Omega_{ij}$ represent all edges from vertex $i$ to vertices $j>i$.
Then $G(n,p)=\prod_{i=1}^{n-1}\Omega_i$. Note here that the vertices are ordered to describe the product space in a way that every edge appears exactly once. Apart from that, this ordering plays no role.

\lateproof{Theorem~\ref{thm:forests}}
Fix $\eps>0$, a tree~$T$, and assume $d_0$ is sufficiently large. Let $L=|T|$ and $d=np$.
A standard first moment computation shows that \textbf{whp} there is no induced $T$-matching of order at least $(2+\eps)\log_q (np)$. We include the short computation for the sake of completeness.
Let $r=(2+\eps)\log_q (np)/L$ and let $Z$ be the number of (unlabelled) induced $T$-matchings in $G(n,p)$ with $r$ components. 
Then we have 
\begin{align*}
\expn{Z} \le \frac{n^{rL}}{r!} p^{r(L-1)} (1-p)^{\binom{rL}{2}-r(L-1)} \le (np)^{rL}(rp/\eul)^{-r}q^{-(rL)^2/2 + 2rL}.
\end{align*}
Using $p\le 0.99$, we see $\log q=\Theta(p)$ and $q=O(1)$. In particular, $rp/\eul =\Omega\left(\frac{\log d_0}{L}\right)\ge 1$ and $rL=\Theta\left(\frac{\log (np)}{p}\right)=\omega_n(1)$. Hence, 
\begin{align*}
	\expn{Z} \le \left(O(1) npq^{-rL/2} \right)^{rL} \le \left(O(1) d_0^{-\eps/2} \right)^{rL} \le 2^{-rL} = o_n(1).
\end{align*}
By Markov's inequality, \textbf{whp} we have $Z=0$.

We now turn to the lower bound.
Let $X$ be the maximum order of an induced $T$-matching in $G(n,p)$.
Our goal is to show that $X\ge (2-\eps)\log_q d$ \textbf{whp}. Set $b=(2-\eps/2)\log_q d$. %(Since $\log_q(d)=\omega_n(1)$, we may assume that $b$ is divisible by~$L$.)

First, by Lemma~\ref{lem:2ndmoment}, we have 
\begin{align*}
\prob{X\ge b} \ge \exp\left(-10^4L^2\frac{n\log^2 d}{d^2} -2d^{-\Omega(\eps)}\right).
\end{align*}
This means that in the case $d\ge n^{1/2}\log^{2}n$, we are already done. Assume now that $d\le n^{1/2}\log^{2}n$. Then the above bound simplifies to 
\begin{align}
\prob{X\ge b} \ge \exp\left(-\frac{n\log^5 d}{d^2}\right).\label{2nd moment}
\end{align}
Recall also that in the regime $d=o(n)$ we have $\log_q d\sim \frac{n}{d}\log d$.

It is easy to check that $X$ is $L$-Lipschitz and $f$-certifiable, where $f(s)=s+L$. Indeed, adding or deleting edges arbitrarily at one vertex can change the value of $X$ by at most~$L$, hence $X$ is $L$-Lipschitz. Moreover, if $X\ge s$, this means there is a set $I\In[n]$ of size $s\le |I|< s+L$ which induces a $T$-matching. If we leave the coordinates indexed by $I$ unchanged, this means in particular that $I$ still induces a $T$-matching, hence we still have $X\ge s$.

Hence, Talagrand's inequality applied with $t=\frac{\sqrt{n}\log^3 d}{d}$ yields
$$\prob{X\le b-tL\sqrt{b+L}} \prob{X\ge b} \le \exp\left(-\frac{n\log^6 d}{4d^2}\right).$$
Together with~\eqref{2nd moment} and since 
\begin{align}
tL\sqrt{b+L}\le \frac{\sqrt{n}\log^3 d}{d} L \sqrt{2\frac{n}{d}\log d}\le \frac{n}{d},\label{Lipschitz}
\end{align}%(where the last inequality justifies our choice of $L$),
we infer that the probability of $X\le b-\frac{n}{d}$ is at most $\exp\left(-\frac{n\log^6 d}{5d^2}\right)=o_n(1)$.
This completes the proof since $b-\frac{n}{d}\ge (2-\eps)\log_q d$.
\endproof

In the above proof, we had some room to spare in~\eqref{Lipschitz}. We will now exploit this to allow the component sizes to grow with~$d$. The proof is almost verbatim the same, so we only point out the differences.

\lateproof{Lemma~\ref{lem:forests}}
Note that we are only interested in the case $d\le n^{1/2}\log^{2}n$ and when $T$ is a path of order~$L$. Since $\Delta(T)$ is bounded, Lemma~\ref{lem:2ndmoment} still provides the lower bound in~\eqref{2nd moment}.
All we have to ensure now is that~\eqref{Lipschitz} still holds, and this is easily seen to be the case as long as $L\le d^{1/2}/\log^4 d$.
\endproof

\section{Connecting}\label{sec:connect}

In this section, we use Lemma~\ref{lem:forests} to prove Theorem~\ref{thm:path} as outlined in Section~\ref{sec:paths}. Recall that we intend to define an auxiliary digraph on the components of a linear forest, where an edge corresponds to a suitable connection between two paths. 
Our goal is to find an almost spanning path in this random digraph. The tool which enables us to achieve this, Lemma~\ref{lem:DFS} below, is based on the well-known graph exploration process \emph{depth-first-search} (DFS).
The usefulness of DFS to find long paths in random graphs was demonstrated impressively by Krivelevich and Sudakov~\cite{KS:13} in a paper where they give surprisingly short and elegant proofs of classical results in random graph theory. For instance, a straightforward consequence of DFS is the following: If an $n$-vertex graph $G$ has the property that any two disjoint sets of size $k$ are joined by an edge, then $G$ contains a path of length $n-2k+1$. The condition needed here can be conveniently checked in random graphs.

In order to connect the paths of the linear forest, we will use some new vertices. For the final path to be induced, we require these new vertices to form an independent set. One potential and clean way to guarantee this is to first find an independent set and then to only use vertices from this set as connecting vertices. However, this reduces the number of potential connecting vertices by a factor of roughly $d$, which is too costly.
Instead, we develop a variant of DFS that can encode conflicts, and use it to find a sufficient condition for the existence of a long ``conflict-free'' path.

We now introduce the notation we need.
Given a set $E$ (which in our case will be the edge set of the auxiliary digraph $D$), a \defn{conflict system} on $E$ is a graph $C$ together with an assignment $\Lambda\colon E \to 2^{V(C)}$. We say that a subset $E'\In E$ is \defn{admissible} (with respect to $C,\Lambda$) if one can select a representative from $\Lambda(e)$ for all $e\in E'$ such that the chosen representatives are distinct and form an independent set in~$C$.
We say that an element $y\in V(C)$ has no conflict with $X\In V(C)$ if there is no edge in $C$ between $y$ and~$X$.

\begin{lemma}\label{lem:DFS}
	Let $G$ be a digraph on $n$ vertices and $C,\Lambda$ a conflict system on~$E(G)$. Suppose that, for any two disjoint sets $S,T\In V(G)$ of size $k$ and any subset $X\In V(C)$ of size at most~$n-1$, there exists an edge $e\in E(G)$ from $S$ to~$T$ and a representative in $\Lambda(e)\sm X$ which has no conflict with~$X$. Then $G$ contains an admissible path of length $n-2k+1$.
	%and an admissible cycle of length at least $n-4k+4$.
\end{lemma}

\begin{proof}
	We proceed in a depth-first-search manner. As usual, we maintain three sets of vertices: the set $S$ of vertices whose exploration is complete, the set $T$ of unvisited vertices, and the set $U=V(G)\sm(S\cup T)$ which functions as a stack (last in, first out). Additionally, we keep track of the set $X\In V(C)$ of chosen representatives. Initially, we have $S=U=X=\emptyset$ and $T=V(G)$. 
	
	In each round, the algorithm proceeds as follows. If the stack $U$ is empty, then some vertex is removed from $T$ and pushed into~$U$. If the stack $U$ is non-empty, consider the last vertex $u$ that was inserted into~$U$. If there exist $v\in T$ such that $uv\in E(G)$ and $y\in \Lambda(uv)\sm X$ such that $y$ has no conflict with $X$, then delete (one such) $v$ from $T$ and insert it into $U$, and add $y$ to~$X$. Otherwise, delete $u$ from $U$ and add it to~$S$.
	The algorithm terminates when $U=T=\emptyset$ and $S=V(G)$.
	
	Clearly, in each round, exactly one vertex ``moves'', either from $T$ to $U$ or from $U$ to~$S$. 
	Hence, there exists a moment when $|S|=|T|$. Moreover, the vertices in $U$ always form an admissible directed path by construction. 
	Suppose now, for the sake of contradiction, that $G$ does not contain an admissible path of the desired length. Then $|U|\le n-2k+1$ and hence $|S|=|T|\ge k$. By assumption of the lemma, there exist $u\in S$ and $v\in T$ such that $uv\in E(G)$ and $y\in \Lambda(uv)\sm X$ such that $y$ has no conflict with~$X$. 
	However, this contradicts the fact that $u$ was moved from $U$ to $S$ at some point, since instead the algorithm would have moved $v$ from $T$ to $U$ and added $y$ to~$X$.
	%(Note here that $X$ collects all chosen representatives for vertices that were moved from $T$ to $U$, not just of those that are currently in~$U$.)
%	This completes the proof of the path case. Using the assumption of the lemma again, we can find a suitable edge from the last $k$ vertices to the first $k$ vertices of an admissible path, which results in an admissible cycle of the desired length.
\end{proof}

Now, we use Lemma~\ref{lem:DFS} to connect the components of an induced linear forest obtained via Lemma~\ref{lem:forests}.

\lateproof{Theorem~\ref{thm:path}}
Fix $\eps>0$ and assume that $d\ge d_0$ is sufficiently large.
We will assume that $d\le n^{1/2}\log^2 n$. For the case $d=\omega(n^{1/2}\log n)$, Lemma~\ref{lem:2ndmoment} implies that \textbf{whp} there exists an induced path even of the asymptotically optimal length $(2-\eps)\frac{n}{d}\log d$.

	Split the vertex set $[n]$ into $V_1$ and $V_2$ each of size at least $n/3$. We explore the random edges of $G\sim G(n,p)$ in two stages. First, we expose the edges inside~$V_1$. Here, we find an induced linear forest $F$ with large components. In the second round, we expose the remaining edges. Our goal is to use some vertices from $V_2$ to connect almost all of the components of $F$ into a large induced path. 
	Set $$L=d^{1/2}/\log^5 d, \quad m = \eps L/8, \quad k=(3/2-\eps/4)\frac{n}{d}\log d, \quad N=k/L.$$
	
	Expose first the edges inside~$V_1$. By Lemma~\ref{lem:forests}, \textbf{whp} we can find in $G[V_1]$ an induced linear forest $F$ with $N$ components of order~$L$.\footnote{We could get an even larger forest from Lemma~\ref{lem:forests}, but this would not help us here since the bottleneck is to ensure that the new connecting vertices only have two edges to the given forest.} Note here that the edge probability $p$ in the application is the same, and $\log(np/3)\sim \log(d)$. 
	
	From now on, we assume that any such $F$ is given. It suffices to prove that, when exposing the edges between $V_1,V_2$ and inside $V_2$, \textbf{whp} we can find the desired induced path.
	
	Let $\cP$ be the set of components of~$F$. We give every path $P\in \cP$ an arbitrary direction, which will be fixed for the rest of the proof. 
	Let $P^-$ denote the first $m$ vertices on $P$, and $P^+$ the last $m$ vertices on~$P$, according to the chosen direction.
	
	We define a (random) auxiliary digraph $D$ with vertex set~$\cP$, where an edge $(P_1,P_2)$ represents a suitable connection between $P_1^+$ and~$P_2^-$.
	Formally, for distinct $P_1,P_2\in \cP$ and $a\in V_2$, we say that $(P_1,P_2)$ is \defn{$a$-connected} if $a$ has exactly one edge to both $P_1^+$ and $P_2^-$, but no edge to any vertex in $V(F)\sm (P_1^+\cup P_2^-)$. The pair $(P_1,P_2)$ forms an edge in $D$ if it is $a$-connected for some $a\in V_2$.
	
	In order to facilitate the application of Lemma~\ref{lem:DFS}, we need to specify a suitable conflict system on~$E(D)$.
	The conflict graph is simply the random graph $G[V_2]$, and to an edge $(P_1,P_2)$ of $D$, we assign the set of all $a\in V_2$ for which $(P_1,P_2)$ is $a$-connected.
	Clearly, an admissible path in $D$ yields an induced path in~$G$.
	To complete the proof, it suffices to show that \textbf{whp} $D$ contains an admissible path of length $(1-\eps/4)N$, as then the induced path in $G$ has length at least $(1-\eps/4)N(L-2m)=(1-\eps/4)(1-\eps/4)NL\ge (3/2-\eps)\frac{n}{d}\log d$, as desired.

	To achieve this, we show that the conditions of Lemma~\ref{lem:DFS} hold \textbf{whp}.
	Fix any disjoint sets $S,T\In \cP$ of size $\eps N/8$ and any $X\In V_2$ of size at most~$N-1$. Ultimately, we want to use a union bound, so we note that the number of choices for these sets are at most $2^N$ for each of $S$ and $T$, and at most $$\sum_{j=0}^{N-1}\binom{n}{j} \le N\binom{n}{N}\le N \left(\frac{\eul n}{N} \right)^N \le \exp(4N\log d)$$ for $X$, where we used $N\le n/2$ in the first and $N\ge n/d^2$ in the last inequality. 
	
	Call $a\in V_2\sm X$ \defn{good} if some pair in $S\times T$ is $a$-connected and $a$ has no conflict with~$X$.
	Observe that $S,T,X$ satisfy the condition of Lemma~\ref{lem:DFS} if and only if some $a\in V_2\sm X$ is good.
	For $(P_1,P_2)\in S\times T$ and $a\in V_2\sm X$, the probability that $(P_1,P_2)$ is $a$-connected and $a$ has no conflict with $X$ is exactly $$\alpha=m^2p^2(1-p)^{|F|-2+|X|}$$ as there are $m^2$ choices for the two neighbours of $a$ in $P_1^+$ and $P_2^-$, and in any such case we need $|F|-2+|X|$ ``non-edges''.
    Since $|F|+|X| \le (3/2-\eps/5) \frac{n}{d}\log d$ and $1-p\ge \exp(-p-O(p^2))$, we have $$(1-p)^{|F|-2+|X|}\ge \exp(-(1+O(p))p(3/2-\eps/5)p^{-1}\log d )) \ge d^{-3/2+\eps/6}.$$ 
    Moreover, two distinct pairs $(P_1,P_2)$, $(P_1',P_2')$ cannot both be $a$-connected at the same time, thus the probability that $a$ is good 
    is simply $|S||T|\alpha$. 
    Finally, whether $a$ is good is determined solely by the potential edges between $a$ and $V(F)\cup X$. Therefore, these events are independent for distinct $a$'s. 
    Hence, the probability that $S,T,X$ violate the condition of Lemma~\ref{lem:DFS}, that is, no $a$ is good, is at most 
    $$(1-|S||T|\alpha)^{|V_2\sm X|} \le \exp(-|S||T|\alpha|V_2\sm X|) \le \exp(-(\eps N/8)^2 \alpha(n/4) ) \le \exp(-Nd^{\eps/7})  ,$$
    where the last inequality holds since
    $$\frac{\eps^2}{256}N\alpha n \ge \frac{n}{dL} L^2 (d/n)^2 d^{-3/2+\eps/6}n = Ld^{-1/2+\eps/6}\ge d^{\eps/7}.$$
	The said union bound completes the proof.
\endproof

\section{Concluding remarks}

\begin{itemize}

\item We proved that the random graph $G(n,d/n)$ \textbf{whp} contains an induced path of length $(3/2-o_d(1))\frac{n}{d}\log d$. It would be very nice to improve the constant $3/2$ to $2$, which would be optimal. One possible way to achieve this, using parts of our argument, is to show that there exists an induced linear forest of size $\sim 2\frac{n}{d}\log d$ where each component path has length $d^{1-o(1)}$.
%This in turn might be achievable using a similar combination of second moment method and concentration inequalities.

\item Our proof is not constructive, since the first part of the argument uses the second moment method.
The previously best bound $\sim \frac{n}{d}\log d$ due to \L{}uczak~\cite{luczak:93} and Suen~\cite{suen:92} was obtained via certain natural algorithms.
It seems that this could be a barrier for such approaches. A (rather unsophisticated) heuristic giving evidence is that when we have grown an induced tree of this size, and assume the edges outside are still random, then the expected number of vertices which could be attached to a given vertex of the tree is less than one.
Moreover, such an ``algorithmic gap'' has been discovered for many other natural problems. 
In particular, Coja-Oghlan and Efthymiou~\cite{COE:15} proved that the space of independent sets of size $k$ becomes ``shattered'' when $k$ passes $\sim \frac{n}{d}\log d$, which seems to cause local search algorithms to get stuck.

\item 
In~\cite{CDKS:ta} it is conjectured that one should not only be able to find an induced path of size $\sim 2\frac{n}{d}\log d$, but any given bounded degree tree. For dense graphs, when $d=\omega(n^{1/2}\log n)$, this follows from the second moment method (see~\cite{draganic:20}). In fact, Lemma~\ref{lem:2ndmoment} shows that the maximum degree can even be a small polynomial. On the contrary, the sparse case seems to be more difficult, mainly because the vanilla second moment method does not work.
However, Dani and Moore~\cite{DM:11} demonstrated that one can actually make the second moment method work, at least for independent sets, by considering a \defn{weighted} version. This even gives a more precise result than the classical one due to Frieze~\cite{frieze:90}. It would be interesting to find out whether this method can be adapted to induced trees.

\end{itemize}

\section*{Acknowledgement}

Thanks to Benny Sudakov and Nemanja Dragani\'c for very useful discussions.

\bibliographystyle{amsplain_v2.0customized}
\bibliography{References}

\end{document}